\documentclass[11pt]{amsart}

\usepackage{amssymb, graphicx}
\usepackage{bbm} % for the complex imaginary unit i

\renewcommand{\emptyset}{\varnothing}

\usepackage{enumerate,color}

\usepackage{amsfonts}
\usepackage{amsmath}
\usepackage{verbatim}

\usepackage[pagebackref,colorlinks,linkcolor=red,citecolor=blue,urlcolor=blue,hypertexnames=true]{hyperref}

\newcommand{\cC}{\mathcal{C}}

\def\beq{ \begin{equation} }
\def\eeq{ \end{equation} }

\def\bep{\begin{proof}}
\def\eep{\end{proof}}

\def\ben{ \begin{enumerate} }
\def\een{ \end{enumerate} }

\newcommand{\ol}{\overline}
\newcommand{\ksi}{\xi}
\newcommand{\tr}{\mathrm{tr}}

\allowdisplaybreaks[1]

\newtheorem{theorem}{Theorem}[section]
\newtheorem{proposition}[theorem]{Proposition}

\newtheorem{corollary}[theorem]{Corollary}

\theoremstyle{definition}

\newtheorem{lemma}[theorem]{Lemma}

\newtheorem{remark}[theorem]{Remark}

\newtheorem{example}[theorem]{Example}

\newcommand{\CC}{\mathbb{C}}
\newcommand{\R}{\mathbb{R}}

\newcommand{\N}{\mathbb{N}}

\newcommand{\X}{\langle X\rangle}
\newcommand{\F}{\mathbb{F}}
\newcommand{\sS}{\mathcal{S}}
\newcommand{\wt}{\widetilde}
\newcommand{\ti}{\tilde}
\newcommand{\cyc}{\stackrel{\mathrm{cyc}}{\thicksim}}
\newcommand{\g}{\sigma} 

\title[A tracial Nullstellensatz]{A tracial Nullstellensatz}

\author[I. Klep]{Igor Klep${}^{1}$}
\address{Igor Klep, Department of Mathematics, 
The University of Auckland, New Zealand}
\email{igor.klep@auckland.ac.nz}
\thanks{${}^1$Supported by the Faculty Research Development Fund (FRDF) of The
University of Auckland (project no. 3701119). Partially supported by the Slovenian Research Agency grant P1-0222.}

\author[\v S. \v Spenko]{\v Spela \v Spenko${}^{2}$}
\address{\v Spela \v Spenko,  Institute of  Mathematics, Physics, and Mechanics,  Ljubljana, Slovenia} \email{spela.spenko@imfm.si}
\thanks{${}^{2}$Supported by the Slovenian Research Agency and in part by the Slovene Human Resources Development and Scholarship Fund.}

\date{\today}

\begin{document}

\begin{abstract}
The main result of this note is a tracial Nullstellensatz for free noncommutative polynomials evaluated
at tuples of matrices of all sizes: Suppose $f_1,\ldots,f_r,f$ are free polynomials, and
$\tr(f)$ vanishes whenever all $\tr(f_j)$ vanish. Then either $1$ or $f$ is a linear combination of the
$f_j$ modulo sums of commutators.
\end{abstract}
 
 \iffalse
 %%%%%%%%%% text-only abstract
The main result of this note is a tracial Nullstellensatz for free noncommutative polynomials evaluated at tuples of matrices of all sizes: Suppose f_1,...,f_r,f are free polynomials, and tr(f) vanishes whenever all tr(f_j) vanish. Then either 1 or f is a linear combination of the f_j modulo sums of commutators.
%%%%%%%%%%%%%%%%
 \fi

\maketitle

\section{Introduction}

Hilbert's Nullstellensatz is a classical result in
algebraic geometry. Over an algebraically closed field 
it characterizes polynomials vanishing on the zero set of a set of polynomials.
Due to its importance it has been generalized and extended in 
many different directions, including to free algebras.
For instance, 
Amitsur's Nullstellensatz
\cite{Ami} describes free noncommutative polynomials vanishing
on the zero set of a given finite set of free polynomials in
a full matrix algebra. 
In another direction,  the
Nullstellensatz of Bergman \cite{HM}
studies a weaker, directional notion of vanishing but in a dimension-independent context 
(see \cite{CHMN13} for recent  generalizations) allowing for a stronger conclusion.
Namely, unlike in Hilbert's and Amitsur's Nullstellensatz,
no powers are needed in the obtained algebraic certificate.
We also refer the reader to \cite{BK11} for a survey of free Nullstellens\"atze.

\smallskip

 In this short article we  
focus on \emph{vanishing trace} of free noncommutative
polynomials.
The relationship between sums of commutators and vanishing trace of a
free polynomial is discussed e.g.~in \cite{CGM,KS,BK11}.
 Our main result, Theorem \ref{spur}, characterizes free polynomials $f$
whose trace vanishes whenever the traces of polynomials $f_1,
\ldots, f_r$ vanish. It is presented in Section \ref{sec:main},
after preliminaries in  Section \ref{sec:not}.
The main ingredients
in the proof of our main result are effective degree bounds
on Hilbert's Nullstellensatz due to Koll\'ar \cite{Kol} (see also Sombra \cite{Som} and Jelonek \cite{Jel}),
as well as the theory of polynomial identities \cite{Row,Pro}.
Finally, in Section \ref{sec:mp} we 
solve a tracial moment problem  by dualizing the statement
of Theorem \ref{spur}.

\section{Preliminaries}\label{sec:not}

\subsection{Notation}
Let $\F$ be a field of characteristic $0$ and let $M(\F)$ stand for $\bigcup_n  M_n(\F)$. We denote the free associative algebra in the variables $x_1,\dots,x_g$ by $\F\X$. 
The free noncommutative polynomials in $\F\X$ of degree at most $d$ are denoted by $\F\X_d$, 
while $\F\X_d'$ is the vector subspace of $\F\X_d$ consisting of all elements with zero constant term. 
We denote by $\X$ the  monoid generated by $x_1,\dots,x_g$, and by $\X_d$ words in $\X$ of degree at most $d$. 

We say that polynomials $f,h\in \F\X$ are {\em cyclically equivalent} if $f-h$ is a sum of commutators in $\F\X$ and write $f\cyc h$.
\subsection{An effective Nullstellensatz}
Let us recall an effective version of Hilbert's Nullstellensatz, 
giving bounds on the polynomials needed in the B\'ezout identity. 
We present a variant that combines \cite{Kol,Jel}. %(see \cite[Corollary 1.7]{K}). 
%%%%%%%%%%%%%%%%%%%%%%%
\begin{comment}
\begin{theorem}
Given $f_1\dots,f_k$ and $h\in K[x_1,\dots,x_n]$, assume that $h$ vanishes on all common zeros of $f_1,\dots,f_k$ (in the algebraic closure of $K$). Let $d_i=\deg(f_i)$ and assume that none of the $d_i$ is $2$. Then one can find $g_1,\dots,g_k\in K[x_1,\dots,x_n]$ and a natural number $s$ satisfying 
$$\sum g_if_i=h^s$$
such that 
$$s\leq N'(n,d_1,\dots,d_k) \text{ and } \deg g_if_i\leq (1+\deg h) N'(n,d_1,\dots,d_k),$$
where 
$$
N(n,d_1,\dots,d_k)=\left\{\begin{array}{ll}
d_1\cdots d_k & \textrm{if $k\leq n$},\\
d_1\cdots d_{n-1}d_k & \textrm{if $k>n>1$},\\
d_1+d_k-1 & \textrm{if $k>n=1$}
\end{array} \right.
$$
for $d_1\geq \dots\geq d_k$.
\end{theorem}
\end{comment}
%%%%%%%%%%%%%%%%%%%%%%
Define
\[
\begin{split}
N(n,d_1,\dots,d_r)&=\left\{\begin{array}{ll}
d_1\cdots d_r & \textrm{if $r\leq n$},\\
d_1\cdots d_{n-1}d_r & \textrm{if $r>n>1$},\\
d_1+d_r-1 & \textrm{if $r>n=1$},
\end{array} \right.
\\
N'(n,d_1,\dots,d_r)&=\left\{\begin{array}{ll}
N(d_1,\dots, d_r) & \textrm{if $r\leq n$},\\
N(d_1,\dots, d_r) & \textrm{if $r>n\geq 1$ and $d_r>2$},\\
2N(d_1,\dots, d_r)-1 & \textrm{if $r>n>1$ and $d_r\leq 2$},\\
%d_1+d_r-1 & \textrm{if $r>n=1$ and $d_r>2$},\\
2d_1-1 & \textrm{if $r>n=1$ and $d_r\leq 2$},
\end{array} \right.
\end{split}
\]
for $d_1\geq \dots\geq d_r$.

\begin{theorem}[Koll\'ar--Jelonek]\label{JelKol}
Let $\F$ be an algebraically closed field and let $f_1, . . . , f_r \in \F[x_1, \dots , x_n]$ be commutative polynomials without a common zero. 
Let $d_i=\deg f_i$ and assume $d_1\geq \dots\geq d_r$.  
Then there exist $h_1,\dots,h_r \in \F[x_1,...,x_n]$ satisfying
$$1 = h_1f_1 + \cdots + h_rf_r,$$ 
with $\deg h_if_i \leq N'(n,d_1,\dots,d_r)$ for $1\leq i\leq r$.
\end{theorem}

A core feature of this theorem we shall use is that the obtained degree bounds are independent of the number of variables $n$ (for large enough $n$).
%%%%%%%%%%%%%%%%%%%%55
\begin{comment}
We will need only the following corollary.
\begin{corollary}
Let $f_1, . . . , f_r \in \F[x_1, \dots , x_n]$ be polynomials without a common zero and let $d_i=\deg f_i$. 
Then there exist  $h_1,\dots,h_r \in \F[x_1,...,x_n]$,  satisfying
$$1 = h_1f_1 + \cdots + h_rf_r,$$ 
with $\deg h_if_i \leq D(d_1,\dots,d_k)$ for a function $D$ depending only on $d_i$, $1\leq i\leq r$.
 \end{corollary}
%Note that for $n>k$ the bound for $s$ does not depend on $n$.
\end{comment}
%%%%%%%%%%%%%%%%%%%%%%%%%%%%%%%%%%%

\subsection{Trace identities}\label{TR}
An {\em algebra with trace}  is an algebra $A$ equipped with an additional structure, that is a linear map $\tr:A\to A$ satisfying the following properties
$$ \tr(ab)=\tr(ba),\quad a\,\tr(b)=\tr(b)\,a,\quad \tr(\tr(a)b)=\tr(a)\tr(b) $$
for all $a,b\in A.$
A morphism between algebras with trace preserves the trace and such algebras form a category. 

 The free algebra in this category is the algebra of free noncommutative polynomials in the variables $x_1,\dots,x_g$ over the polynomial algebra $T$  in the infinitely many variables  $\tr(w)$, where $w$ runs over all representatives of the cyclic equivalence classes of words in the variables $x_1,\dots,x_g$. Its elements are {\em trace polynomials} and elements of $T$ are {\em pure trace polynomials}. 
The degree of a trace monomial $\tr(w_1)\cdots\tr(w_m)v$, $w_i,v\in \X$,  equals $|v|+\sum_i |w_i|$, where $|u|$ denotes the length of a word $u$.
The degree of a trace polynomial is the maximum of the degrees of its trace monomials.

{\em Trace identities} of the matrix algebra $M_n(\F)$ are the elements in the kernel of the evaluation map from the free algebra with trace to $M_n(\F)$. {\em Pure trace identities} are trace identities that belong to $T$. 
By \cite[Theorem 4.5]{Pro}, there are no trace identities of $M_n(\F)$ of degree  less than $n$ and no pure trace identities of degree  less than $n+1$.

\section{Main Result}\label{sec:main}

Throughout this section let $\F$ be an algebraically closed field of characteristic $0$. 
The following is the main result of this paper. It is proved in Subsection \ref{subsec:pf} below.
We remark that a very special case of Theorem \ref{spur} was obtained in
\cite{BK11} by different means.

\begin{theorem}[Spurnullstellensatz]\label{spur}
Let $f_1,\dots,f_r,f\in \F\X$.
%Let $f_1,\dots,f_r,f$ be noncommutative polynomials.
The implication
\begin{equation}\label{trpogoj}
\tr(f_1(A))=\dots=\tr(f_r(A))=0 \quad \implies \quad \tr(f(A))=0
\end{equation}
%$$(\tr(f_1(a_1,\dots,a_g))=0\land \dots \land \tr(f_r(a_1,\dots,a_g))=0)\implies \tr(f(a_1,\dots,a_g))=0$$
holds for every $n$ and all $A\in M_n(\F)^g$ if and only if  %$\sum \lambda_i f_i\sim_{cyc} 1$ or $\sum \lambda_i f_i\sim_{cyc} f$ for some $\lambda_i\in \F$, $1\leq i\leq r$.
$f$ is cyclically equivalent to a linear combination of $f_i$'s or a linear combination of $f_i$'s is cyclically equivalent to a nonzero scalar.
\end{theorem}

\begin{remark}
Note $\sum \lambda_i f_i\cyc1$ for some $\lambda_i\in \F$ does not necessarily imply that $f\cyc \sum \mu_i f_i$ for some $\mu_i\in \F$. For example, take $f_1=1, f=x_1$.
\end{remark}

\begin{example}
Theorem \ref{spur} fails in the dimension-dependent context: For each $n\in \N$ 
we give an example of polynomials $f_1,f_2$ such that $\tr(f_1),\tr(f_2)$ do not have a common zero on $M_n(\F)$, but for which there do not exist
$\lambda_1,\lambda_2\in \F$ and a polynomial identity $p$ of $M_n(\F)$ such that 
\beq\label{eq:central}
\lambda_1 f_1+\lambda_2 f_2\cyc 1+p.
\eeq

Let $c$ be a homogeneous central polynomial of $M_n(\F)$, which means that $c(A)\in \F$ for all $A\in M_n(\F)^g$ and $c$ does not vanish identically on  $M_n(\F)$ (see e.g.~\cite{Row}).
%(A polynomial is central on $M_n(\F)$ if all of its values lie in the centre of $M_n(\F)$ but is not a polynomial identity of $M_n(\F)$ (see, e.g. \cite{Row}).)
Take 
\[f_1=c,\; f_2=1+c^2.\] Then $\tr(f_1), \tr(f_2)$ do not have a common zero on $M_n(\F)$. 
If \eqref{eq:central} holds for some $\lambda_1,\lambda_2\in \F$ and a polynomial identity $p$ of $M_n(\F)$, 
then \[\lambda_1 \tr(c)+\lambda_2\Big(\frac{1}{n}\tr(c)^2+n\Big)=n.\] 
As $c$ is homogeneous, say of degree $k$, then 
\beq\label{eq:homog}
\lambda_1 \alpha^k\tr(c)+\lambda_2\Big(\frac{1}{n}\alpha^{2k}\tr(c)^2+n\Big)=n
\eeq
 for every $\alpha\in \F$. 
But $\F$ is infinite, so \eqref{eq:homog} cannot hold for every $\alpha\in \F$.
\end{example}

\subsection{Proof of Theorem \ref{spur}}\label{subsec:pf}

As a first step towards the proof of Theorem \ref{spur} we
prove its weaker variant, characterizing sets of polynomials $f_i$
whose traces do not have a common vanishing point.

\begin{lemma}\label{sibekspur}
Let  $f_1,\dots,f_r\in \F\X$.
If the equations 
\begin{equation}\label{pogoj0}
\tr(f_i(x_1,\dots,x_g))=0,\:1\leq i\leq r,
\end{equation}
  do not have a common solution in $M(\F)^g$, 
then $\sum \lambda_i f_i\cyc 1$ for some $\lambda_i\in \F$.
\end{lemma}

\begin{proof}
We can assume that $\tr(f_i)$'s are linearly independent as elements in the free algebra with trace 
and we also assume that $\tr(f_i)$ cannot be written as $\tr(f_i')$ for a polynomial $f_i'$ with $\deg(f_i')<\deg(f_i)$. 

%For every fixed $n$ the condition (\ref{pogoj0}) can be regarded in the trace algebra ${\rm R}_n$ of degree $n$. 
%The algebra ${\rm GM}_n$ of generic matrices of degree $n$ can be naturally viewed as a subalgebra of $M_n(\F[x_{ij}^{(k)}\mid 1\leq i,j\leq n,1\leq k\leq g])$ and 
For every fixed $n$ we can evaluate a noncommutative polynomial $f$ on the symbolic $n\times n$ matrices $\ksi_k=(x_{ij}^{(k)})$, $1\leq k\leq g$, and 
the trace of $f(\ksi_1,\dots,\ksi_g)$ is a commutative polynomial in the variables $x_{ij}^{(k)}$, $1\leq i,j\leq n$, $1\leq k\leq g$. 
Note that $\deg(\tr(f))\leq \deg(f)$ for every $f\in \F\X$, 
where $\deg(\tr(f))$ denotes the degree  of the commutative polynomial $\tr(f(\ksi_1,\dots,\ksi_g))$.
By Theorem \ref{JelKol}, the condition (\ref{pogoj0}) implies  that there exist $h_1^{(n)},\dots, h_{r}^{(n)}\in \F[x_{ij}^{(k)}:1\leq i,j\leq n, 1\leq k\leq g]$ such that
\begin{equation}\label{predR}
1=\sum h_i^{(n)}\tr(f_i)
\end{equation}
with 
\begin{equation}\label{meje}
\deg(\tr(f_i))\deg(h_i^{(n)})\leq N'(d_1,\dots,d_r,n^2g),
\end{equation}
 where $d_i=\deg(f_i)$ and $d_1\geq \dots\geq d_r$. 
Applying the Reynolds operator 
(for the usual action of ${\rm GL_n}$ on the polynomial ring $\F[x_{ij}^{(k)}:1\leq i,j\leq n, 1\leq k\leq g]$; i.e., 
$\g\in {\rm GL}_n$ sends a variable $x_{ij}^{(k)}$ into the $(i,j)$-entry of the matrix 
$\g^{-1}\ (x_{ij}^{(k)})\ \g$)
%$\g \cdot x_{ij}^{(k)}=(\g^{-1}\ (x_{ij}^{(k)})\ \g)_{ij}$ for $\g\in {\rm GL}_n$) 
to the equality (\ref{predR}) we can assume that the $h_i^{(n)}$ are invariant, 
and thus pure trace polynomials by \cite[Theorem 1.3]{Pro}. 
 The above bound (\ref{meje})  is independent of $n$ for %$n\geq\frac{\deg(f_i)}{2}$ and
 $r\leq n^2 g$ (see Theorem \ref{JelKol}). % for every $n$ and accordingly change $h_i^{(n)}$. 
 Thus, the degree of $h_i^{(n)}\tr(f_i)$ can be in this case bounded above by $d_1\cdots d_r$.% for $n^2\geq \frac{k}{g}$ and $n>\frac{\deg(f_i)}{2}$. 

We thus get for every sufficiently large  $n$; i.e, for $n\geq d_1\cdots d_r$ and $n\geq\sqrt\frac{r}{g}$, a trace identity for $M_n(\F)$,
\beq\label{eq:trId}
1=\sum h_i^{(n)}\tr(f_i),
\eeq
where $\deg  (h_i^{(n)}\tr(f_i))\leq n$. 
%(In particular, by the assumption in the beginning of the proof $\deg(f_i)<n$.)
Let us fix $n$ with these properties.
Since any nontrivial  pure trace identity on $n\times n$ matrices has degree at least $n+1$ by \cite[Theorem 4.5]{Pro}, the above identity \eqref{eq:trId} must be trivial, which means that it holds in the free algebra with trace. 
A little care is needed at this point. 
If $f_1,\dots,f_r$ do not all have  zero constant term, then 
\eqref{eq:trId}
is an identity in the free algebra with trace if we replace $\tr(f_i)$ by \[ \tau_i=\tr(\ol{f_i})+\alpha_i n, \quad 1\leq i\leq r,\] %$\tr(f)$ by $\tau=\tr(\ol{f})+\alpha n$, 
where $\ol{f_i}$ is the sum of all nonconstant terms of  $f_i$, and $\alpha_i$ is its  constant term, and thus obtain an identity
\begin{equation}\label{tau}
1=\sum h_i^{(n)}\tau_i.
\end{equation} 
 %, for some sufficiently large $n$.)
Since all polynomials that appear in \eqref{tau} are pure trace polynomials, they belong to the commutative subalgebra $T$ of the free algebra with trace, which is generated by the trace monomials $\tr(x_{i_1}\cdots x_{i_t})$ modulo the relations $\tr(x_{i_1}x_{i_2}\cdots x_{i_t})=\tr(x_{i_2}\cdots x_{i_t}x_{i_1})$.
Let $w_1,w_2,\dots$ be the representatives of the cyclic equivalence classes of words in the variables $x_1,\dots,x_g$.
Then $T$ is the free commutative algebra generated by $t_1,t_2,\dots$, where $t_i=\tr(w_i)$.

Let us  denote by $t_0$ the empty word; i.e., the identity of $T$, and write 
$$\tau_i=\sum_{j=1}^m \alpha_{ij}t_j+\alpha_{im+1}t_0.$$
Note that $\alpha_{im+1}=\alpha_i n$. %, $\alpha_{r+1m+1}=\alpha n$. 
As we assumed that $\tr(f_i)$ are linearly independent, %and none of the linear combination of $f_i$ is cyclically equivalent to a nonzero scalar,
 also $\tau_i$ are linearly independent. 
Indeed, assume that $\sum \lambda_i\tau_i=0$ for some $\lambda_i\in \F$. 
Then $\sum \lambda_i\tr(f_i)=0$ on $M_n(\F)$ for the chosen $n\geq\deg(f_i)$.
Thus, $\sum\lambda_if_i$ has zero trace on $M_n(\F)$, so it is cyclically equivalent to a polynomial identity on $M_n(\F)$ \cite[Theorem 4.5]{BK09}.
As $n\geq\deg(f_i)$ we can conclude that $\sum \lambda_i \tr(f_i)=0$ is an identity in the free algebra with trace  and due to the assumption, $\lambda_i=0$ for all $i$.

Note that this implies that the matrix $(\alpha_{ij})$ has linearly independent rows and can be brought to the reduced row echelon form. In particular,
 we can express $t_{i_1},\dots,t_{i_{r}}$ as linear combinations of $\tau_1,\dots,\tau_r$ and $t_{i_{r+1}},\dots,t_{i_m}$, for some $\{i_0,i_1,\dots,i_m\}=\{0,1,\dots,m\}$. 
If the last row  in the reduced echelon form of this matrix has all zeros except for the last entry, then
$1=\sum \lambda_i \tau_i$. In this case $\tr(\sum \lambda_i f_i-\frac{1}{n})=0$ on $M_n(\F)$, which implies that $\sum \lambda_if_i$ is cyclically equivalent to a nonzero scalar. 
Otherwise we can choose (free) generators of $T$ that include $\tau_1,\dots,\tau_r$ and the above identity \eqref{tau} cannot hold since it does not hold in the free commutative algebra.
\end{proof}

\begin{proof}[Proof of Theorem {\rm\ref{spur}}]
%If $f$ is cyclically equivalent to a linear combination of $f_i$'s or a linear combination of $f_i$'s is cyclically equivalent to a nonzero scalar
To prove the nontrivial direction, assume that  no linear combination of $f_i$'s is cyclically equivalent to a nonzero scalar. 
The condition \eqref{trpogoj} implies that the equations 
$$\tr(f_i(x_1,\dots,x_g))=0\;\;(1\leq i\leq r),\quad \tr(f(x_1,\dots,x_g)+1)=0$$ 
do not have a common solution in $M(\F)$. 
Hence \[\sum \lambda_i f_i+\lambda (1+f)\cyc 1\] for some $\lambda_i,\lambda\in \F$ by Lemma \ref{sibekspur}.
By our assumption, $\lambda\neq 0$. 
Thus, 
\[f\cyc \sum \mu_i f_i+\mu \] for some $\mu_i,\mu\in \F$.
If $\mu\neq 0$, then the initial condition \eqref{trpogoj} is violated again by our assumption and Lemma \ref{sibekspur}. Hence $\mu=0$ and $f\cyc \sum \mu_i f_i$.
\end{proof}

\subsection{Bounds on the size of matrices in Theorem \ref{spur}}

The proof of Lemma \ref{sibekspur} reveals a bound on the size of matrices for which it suffices to test the condition of this lemma and of Theorem \ref{spur} in order to draw the conclusion.
If the implication (\ref{trpogoj}) %$\tr(f_1(A))=\dots=\tr(f_r(A))=0\implies \tr(f(A))=0$ 
 holds for all 
$A\in M_N(\F)^g$ for $N=\max\Big\{d_1\cdots d_rd,\sqrt\frac{g}{r}\Big\}$, 
where $d_i=\deg f_i$, $d=\deg f$, then it holds for all $A\in M_n(\F)^g$ for all $n\in \N$.
In view of Theorem \ref{JelKol} we can sometimes sharpen this bound in the case that there exists $m$ satisfying 
$d_1\cdots d_{m^2g}d_r\leq m<\sqrt\frac{g}{r}$.

\subsection{Passing between a real closed field and its algebraic closure}

\begin{proposition}\label{real}
Let $R$ be a real closed field {\rm (}e.g.~$R=\R${\rm )} and let $C$ be its algebraic closure. For polynomials $f_1,\dots,f_r,f\in R\X$ 
the following conditions are equivalent:
\begin{enumerate}[\rm (i)]
\item\label{it:1} 
For every $n\in \N$ and all $A\in M_n(C)^g$ we have
$\tr(f_1(A))=\dots=\tr(f_r(A))=0$ implies $\tr(f(A))=0$;
\item\label{it:2} There exist $\lambda_i\in C$ such that $\sum \lambda_i f_i\cyc 1$ or $\sum \lambda_i f_i\cyc f$;
\item\label{it:3} There exist $\lambda_i\in R$ such that $\sum \lambda_i f_i\cyc 1$ or $\sum \lambda_i f_i\cyc f$;
\item\label{it:4} For every $n\in \N$ and all $A\in M_n(R)^g$ we have $\tr(f_1(A))=\dots=\tr(f_r(A))=0$ implies $\tr(f(A))=0$.
\end{enumerate}
\end{proposition}

\begin{proof}
By Theorem \ref{spur}, \eqref{it:1} is equivalent to \eqref{it:2}. 
By taking real parts, it is easy to see that \eqref{it:2} implies \eqref{it:3}.
The implication \eqref{it:3} to \eqref{it:4} is trivial.
We will prove that \eqref{it:2} follows by assuming \eqref{it:4}.
To obtain a contradiction suppose first that the equations 
\beq\label{eq:commonR}
\tr(f_i(x_1,\dots,x_g))=0, \; 1\leq i\leq r
\eeq
 do not have a common solution in $M(R)$ but do have one in $M(C)$.
Let $a_1,\dots,a_g\in M_n(C)$ be such that $\tr(f_i(a_1,\dots,a_g))=0$ for $1\leq i\leq r$.
Write $a_j=b_j+\mathbbm{i}\,c_j$, where $b_j,c_j\in M_n(R)$, and define
$$
\widetilde{a_j}=\left(
\begin{array}{cc}
b_j&c_j\\
-c_j&b_j
\end{array}
\right).
$$
We have \[\tr(f_i(\wt{a_1},\dots,\wt{a_g}))=2{\rm Re}(\tr(f_i(a_1,\dots,a_g)))=0\] for every $1\leq i\leq r$.
Thus, $\wt{a_1},\dots,\wt{a_g}$ is a common solution of the equations 
\eqref{eq:commonR}
in $M(R)$, a contradiction. 
By Lemma \ref{sibekspur} we therefore have $\sum \lambda_i f_i\cyc 1$ for some $\lambda_i\in C$. 

If the system \eqref{eq:commonR} does have a solution in $M(R)$, 
then \eqref{it:4} implies that the equations 
\[\tr(f_i(x_1,\dots,x_g))=0,\; \tr(1+f(x_1,\dots,x_g))=0\] do not have a common solution in $M(R)$, 
and by the previous step, applied to polynomials $f_1,\dots,f_r,1+f$, 
 they also do not have a solution in $M(C)$. 
By Lemma \ref{sibekspur}, $f\cyc \sum \lambda_i f_i+\lambda 1$ for some $\lambda_i,\lambda\in C$.
Since we are assuming that the equations \eqref{eq:commonR} have a common solution in $M(R)$, 
\eqref{it:4} implies $\lambda=0$.
\end{proof}

\section{A tracial moment problem}\label{sec:mp}

The  main result of this section, Corollary \ref{cor:mp},
solves a constrained truncated tracial moment problem.
For its proof we dualize the statement of Theorem \ref{spur}.
We refer the reader to \cite{Bur} and the references therein for
more details on tracial moment problems.

\smallskip
Let $\F\in \{\R,\CC\}$.
We say that a linear functional $L:\F\X_d\to \F$ is \emph{tracial} if it vanishes on sums of commutators, or equivalently, if $L(v)=L(w)$ for $v\cyc w$.
The simplest examples of such $L$ are obtained as follows.
For $A\in M_n(\F)^g$  define \[\phi_A:\F\X_d\to \F, \quad \phi_A(p)=\tr(p(A)),\]  and let 
$$\cC={\rm span}\{\phi_A\mid A\in M_n(\F)^g,n\in \N\} \subseteq \F\X_d^*.$$

Tracial linear functionals can be described in terms of  moment sequences.
A sequence $(\alpha_w)_{w\in \X_d}$ in $\F$ 
is a {\em truncated tracial moment sequence} if $\alpha_w=\alpha_v$ for $v\cyc w$. 
Note that any element in $\cC$ is tracial and defines a (truncated) tracial moment sequence. 

%[Truncated tracial moment problem]

\begin{proposition}\label{prop:dual}
If $L$ is a tracial linear functional on $\F\X_d$, 
%$(L(w))_{w\in \langle X\rangle_d}$ is a tracial moment sequence, 
then $L\in \cC$.
\end{proposition}

\begin{proof}
If $L\not \in \cC$, then there  exists a linear functional $p\in (\F\X_d^*)^*\cong \F\X_d$ such that
$p(\phi)=0$ for every $\phi\in \cC$ and $p(L)=1$.
By definition of $\cC$ we have $\tr(p(A))=0$ for all $A\in M_n(\F)^g$, $n\in \N$, 
which implies  $p\cyc 0$ (see e.g.~\cite[Corollary 4.4]{Pro}). 
As $L$ is tracial,  $p(L)=L(p)=0$, a contradiction.
\end{proof}

To  give an explicit representation for the tracial linear functional $L$
as a linear combination of the $\phi_A$ with $g$-tuples $A$ of $d\times d$ matrices,  
 we present an alternative constructive  proof of Proposition \ref{prop:dual}.

\begin{proof}[Alternative proof]
Given a tracial moment sequence $(L(w))_{w\in \X_d}$ we  show how to find $g$-tuples $A_{\ell}$ such that 
$L(p)=\sum \lambda_{\ell}\phi_{A_{\ell}}$.
We proceed inductively. 
Assume that there exist $A_{\ell}$, $1\leq {\ell}\leq \ell_m$, such that $L(p)=\sum \phi_{A_{\ell}}$ 
for all $p\in \F\X_m'$. 
We want to find an element of $\cC$ which coincides with $L$ on $\F\X_{m+1}'$. 
Define $L_{m+1}=L-\sum_{\ell=1}^{\ell_m} \phi_{A_{\ell}}$.
It is enough to choose matrices $a_{{\ell}1},\dots,a_{{\ell}g}\in M_{m+1}(\F)$, $\ell_m+1\leq \ell\leq \ell_{m+1}$,
such that $\tr(p(a_{\ell 1},\dots,a_{\ell g}))=0$ for  $p\in \F\X_{m}'$
and $L_{m+1}(p)=\sum_{\ell=\ell_m+1}^{\ell_{m+1}} \tr(p(a_{\ell1},\dots,a_{\ell g}))$ for $p$ homogeneous of degree $m+1$.

Choose representatives $w_{\ell_m+1},\dots,w_{\ell_{m+1}}$ of cyclic equivalence classes of words in the variables $x_1,\dots,x_g$ of degree $m+1$.
Let $w_{\ell}=x_{i_1}^{j_1}\cdots x_{i_s}^{j_s}$, where $\sum_k j_k=m+1$. 
We denote $s_k=\sum_{i=1}^{k} j_i$. 
Setting $a_{\ell i}=0$ at the beginning, we define  matrices $a_{\ell i}\in M_{m+1}(\F)$ $(1\leq i\leq g)$ as follows. 
We let $k$ vary from $1$ to $s$, and at step $k$ we replace $a_{\ell i_k}$ by
\[
%a_{\ell i_k}=
\begin{cases}
a_{\ell i_k}+\sum_{u=s_{k-1}+1}^ {s_k}e_{u,u+1}&\textrm{if $k<s$},\\[.2cm]
a_{\ell i_k}+\sum_{u=s_{k-1}+1}^ {s_k-1}e_{u,u+1}+L_{m+1}(w_\ell)e_{m+1,1}&\textrm{if $k=s$}.
\end{cases}
\]
Here $e_{ij}$ are the standard $(m+1)\times(m+1)$ matrix units.

We claim that the only word in $a_{\ell1},\dots,a_{\ell g}$ of degree $\leq m+1$ with nonzero trace is cyclically equivalent to $w_\ell$. 
A necessary condition for a word $w$ in $a_{\ell1},\dots,a_{\ell g}$, $w=a_{\ell p_1}^{r_1}\cdots a_{\ell p_s}^{r_s}$,
 of degree $m'$, $1\leq m'\leq m+1$, to have  nonzero trace 
is the existence of a sequence $(\ti e_j)_{j=1}^{m+1}$ of matrix units from the set $E=\{e_{u,u+1},e_{m+1,1}\mid 1\leq u\leq m\}$ such that 
$\tr(\tilde e_1\cdots \ti e_{m+1})=1$, and if $\sum_{i=1}^{k-1} r_i<j\leq \sum_{i=1}^{k} r_i$ then $\ti e_j$ appears in $a_{\ell p_k}$.
The product of the elements in $E$ has  nonzero trace only in a unique order (up to  cyclic
permutations).
Since every element in $E$ appears only in one $a_{\ell i}$, this order determines $p_1,\dots,p_s$. 
Thus, $\tr(w(a_{\ell1},\dots,a_{\ell g}))=L_{m+1}(w)$ for $w\cyc w_\ell$ and $0$ otherwise.
Therefore, $a_{\ell1},\dots,a_{\ell g}$, $\ell_m+1\leq \ell\leq \ell_{m+1}$, have the desired properties.

We have thus found $g$-tuples $A_\ell\in M_{n_\ell}(\F)^g$, $1\leq \ell\leq \ell_d$, such that $L(p)=\sum \phi_{A_\ell}(p)$ for every $p\in \F\X_d'$.
Take $A_0=(0,\dots,0)\in \F^g$, and notice that 
$L(p)=\sum \phi_{A_\ell}(p)+(L(\emptyset)-n)\phi_{A_0}(p)$, where $n=\sum \phi_{A_\ell}(1)=\sum n_\ell$, for every $p\in \F\X_d$. 
\end{proof}

Let us fix  polynomials $f_1,\dots,f_r\in\F\X_d$ and write $f_i=\sum\lambda_{ij}w_j$.
We say that a sequence $(L(w))_{w\in \X_d}$ is a {\em constrained truncated tracial moment sequence} if it is a truncated tracial moment sequence and if 
$L(f_i)=\sum_j \lambda_{ij} L(w_j)=0$ for $1\leq i\leq r$.
We define a constrained analog of $\cC$,
$$\mathcal{S}={\rm span}\{\phi_A\in \cC\mid \phi_A(f_i)=0,\; 1\leq i\leq r\}.$$
Note that every element of $\sS$ defines a constrained truncated tracial moment sequence.

\begin{corollary}[Constrained truncated tracial moment problem]
\label{cor:mp}
If $(L(w))_{w\in \langle X\rangle_d}$  is a constrained tracial moment sequence with $L(1)=1$, then $L\in \sS$.
\end{corollary}

\begin{proof}
If $L\not\in \sS$ then there exists an element $p\in(\F\X_d^*)^*$ such that $p(\phi)=0$ for all $\phi\in \sS$ and $p(L)=1$.
We have $\tr(p(A))=0$ for all $A\in M_n(\F)^g$ with the property $\tr(f_i(A))=0$ for all $1\leq i\leq r$.
Thus, Theorem \ref{spur} and Proposition \ref{real} imply that $p\cyc\sum \lambda_i f_i$  or $\sum \lambda_i f_i\cyc 1$ for some $\lambda_i\in \F$.
In the former case we have $L(p)=\sum \lambda_i L(f_i)=0$, which contradicts $L(p)=1$, 
in the last case $L(1)=\sum \lambda_iL(f_i)=0$, which is contrary to the assumption $L(1)=1$.
\end{proof}

\subsubsection*{Acknowledgments.}
This paper was written while the second author was visiting The University of Auckland.
She would like to thank the Department of Mathematics for its hospitality.

\end{document}